\documentclass[12pt]{amsart}

\usepackage{amssymb}
\usepackage{amsfonts}
\usepackage{graphicx}

\newcommand{\Rq}{{\mathbf R}^{4}}
\newcommand{\Rt}{{\mathbf R}^{3}}

\newcommand{\de}{\text{\rm d}\mspace{1mu}}

\newcommand{\fff}{\boldsymbol{I}}
\newcommand{\sff}{\mathbf{I} \mathbf{I}}
\newcommand{\He}{\text{Hess}\mspace{1mu}}
\newcommand{\rank}{\rm rank\,}
\newcommand{\beq}{\begin{equation}}
\newcommand{\eeq}{\end{equation}}
\newcommand{\ba}{\begin{align}}
\newcommand{\ea}{\end{align}}
\newcommand{\baa}{\begin{align*}}
\newcommand{\eaa}{\end{align*}}

\newtheorem{theorem}{Theorem}
\newtheorem{lemma}{Lemma}

\newtheorem{prop}{Proposition}

\newtheorem*{wineq}{Wintgen inequality}

\theoremstyle{definition}
\newtheorem{defn}{Definition}

\theoremstyle{remark}
\newtheorem{rem}{Remark}

\begin{document}

\title{Local geometry of surfaces in $\mathbf R^4$}

\author{J. Basto-Gon\c calves}
\date{\today}
\address{Centro de Matem\'atica da Universidade do
Porto, Portugal}
\email{jbg@fc.up.pt}
\thanks{The research of the  author at Centro de Matem\'atica da Universidade do Porto (CMUP) was
funded by the European Regional Development Funding FEDER through the programme COMPETE and by the Portuguese Government through the FCT -- Funda\c{c}\~ao para a Ci\^encia e a
Tecnologia under the project PEst-C/MAT/UI0144/2011. and  Calouste Gulbenkian
Foundation}

\subjclass{Primary:}
\keywords{}

\begin{abstract}
The indicatrix or curvature ellipse and the characteristic curve of a surface in $\mathbf R^4$ are presented, as well as  the projective duality connecting them. The characterisation of points in the surfaces as elliptic, parabolic and hyperbolic points,  and the inflection points, are also discussed,.
\end{abstract}
\maketitle

\section{Introduction}

For a surface $S$ in $\mathbf R^3$, the Dupin indicatrix  is a conic in the tangent space $T_pS$ at a point $p$ that gives local information on the geometry of the surface, at least at generic points where the conic is non degenerate; the points are hyperbolic or elliptic as the Dupin indicatrix is a hyperbola or a ellipse, or equivalently, as the Gauss curvature is negative or positive, and parabolic when the Gauss curvature vanishes.

For surfaces in $\Rq$ there is no exact analogue of the Dupin indicatrix, but the indicatrix or curvature ellipse and the characteristic curve give a similar type of local information. The indicatrix at a point $p\in S$ is an ellipse in the normal plane $N_pS$ at $p$, and the characteristic curve is a conic, but not necessarily an ellipse, also in the normal plane. A generic point is hyperbolic or elliptic as the characteristic curve is a hyperbola or an ellipse, or as the origin is outside or inside the indicatrix, but the relation with the Gauss curvature is somewhat lost: the curvature is negative at a generic hyperbolic points but it is not always positive, or at least non negative, at elliptic points.

The results discussed here have been known for a long time \cite{kom,mw}, and some of them have been presented in a more contemporary fashion in \cite{little}, and subsequently in \cite{dcc, rdcc}. The objective of this work is to present a more detailed and complete description of the construction of the two conics, the indicatrix and the characteristic curve, and of the relation between them.

Associated to the indicatrix and the characteristic curve there are special normal directions, the binormals, and tangent directions, the asymptotic directions. Their analogy with the similarly named objects in $\Rt$ is best understood in the context of the singularities in the contact of hyperplanes with the surface \cite{dcc}, or of lines with the surface, as presented in the last section.

\newpage
\section{Moving frames}

We consider a surface $S\subset \mathbf R^4$ locally given by a parametrization:
\[
\Xi : U\subset\mathbf R^2\longrightarrow \Rq
\]
and  a set $\{e_1,e_2,e_3,e_4\}$  of orthonormal vectors, depending on $(x,y)\in U$,  satisfying:
\begin{itemize}
\item $e_1(x,y)$ and $e_2(x,y)$ span the tangent space $T_{\Xi(x,y)}S$ of $S$ at $\Xi(x,y)$.
\item $e_3(x,y)$ and $e_4(x,y)$ span the normal space $N_{\Xi(x,y)}S$ of $S$ at $\Xi(x,y)$.
\end{itemize}
Then $\Xi,\{e_1,e_2,e_3,e_4\}$ is an \emph{adapted moving frame} for $S$. Associated to this frame, there is a dual basis for 1-forms, $\{\omega_1,\omega_2,\omega_3,\omega_4\}$.

If we take $U$ small enough, $\Xi$ can be assumed to be an embedding; then the vectors $e_i$ and the 1-forms  $\omega_i$ can be extended to an open subset of $\Rq$.
We define new 1-forms by:
\[
\omega_{ij}=De_i\cdot e_j, \hbox{ also written as } \omega_{ij}=\de e_i\cdot e_j,\quad i,j=1,\ldots,4
\]
where the exterior differential is taken componentwise.

The pullbacks by $\Xi$ are defined by:
\[
\omega_i(v)=D\Xi(v)\cdot e_i,\hbox{ also written as } \omega_i=\de\Xi\cdot e_i
\]
and
\[
\omega_{ij}(v)=De_i(v)\cdot e_j,\quad i,j=1,\ldots,4,\qquad v\in \mathbf R^2
\]
With a slight abuse of notation, we denote the forms on $\Rq$ and their pullbacks to $U$ by the same symbol.
 
The Maurer-Cartan structure equations can be obtained \cite{manfredo}  using $\de\de=0$:
\beq
\de \omega_i=\sum_{j=1}^4 \omega_{ij} \wedge\omega_j,\qquad
\de \omega_{ij}=\sum_{k=1}^4 \omega_{ik} \wedge\omega_{kj}
\eeq

The 1-form $\omega_{12}$ is the connection form for the tangent bundle of $S$, and $\omega_{34}$ is the connection form for the normal bundle of $S$; $\de\omega_{12}$ and 
$\de\omega_{34}$ are the respective curvature forms.
The \emph{Gaussian curvature} $K$ and the \emph{normal curvature} $\kappa$ are defined \cite{little}, respectively, by:
\beq
\de \omega_{12}=-K \omega_1\wedge \omega_2,\qquad
\de \omega_{34}=-\kappa\, \omega_1\wedge \omega_2
\eeq

The forms $\omega_1$ and $\omega_2$ are independent, and $\sigma_S=\omega_1\wedge\omega_2$ is the \emph{area element} on $S$; in fact:
\begin{prop}
The 2-form $\sigma_S=\omega_1\wedge\omega_2$ is independent of the choice of frames, and it is globally defined.
\end{prop}

From $\omega_3=\omega_4=0$ it follows:
\begin{align*}
\de\omega_3=&0=\omega_{31}\wedge \omega_1+\omega_{32}\wedge \omega_2\\
\de\omega_4=&0=\omega_{41}\wedge \omega_1+\omega_{42}\wedge \omega_2
\end{align*}
and by Cartan's lemma \cite{manfredo},  there exist $a$, $b$, $c$, $e$, $f$, and $g$ such that:
\begin{align}\label{ae}
\omega_{13}=&a\mspace{1mu} \omega_1+b\mspace{1mu}\omega_2,& \omega_{14}=&e\mspace{1mu} \omega_1+f\mspace{1mu}\omega_2\\
\notag
\omega_{23}=&b\mspace{1mu} \omega_1+c\mspace{1mu}\omega_2,& \omega_{24}=&f\mspace{1mu} \omega_1+g\mspace{1mu}\omega_2\notag
\end{align}

While the image of $D\Xi$ is the tangent space of $S$, the image of the second derivative $D^2\Xi$ has both tangent and normal components; the vector valued quadratic form associated to the normal component:
\beq
(D^2\Xi\cdot e_3) e_3+(D^2\Xi\cdot e_4) e_4
\eeq
is the \emph{second fundamental form} $\sff$ of $S$. It can be written \cite{little} as $\sff_1e_3+\sff_2 e_4$, where:
\begin{subequations}\label{sff}
\begin{align}
\sff_1=&a\mspace{1mu} \omega_1^2+2b\mspace{1mu}  \omega_1  \omega_2+c\mspace{1mu}  \omega_2^2\label{sff1}\\
\sff_2=&e\mspace{1mu}  \omega_1^2+2f\mspace{1mu}  \omega_1  \omega_2+g\mspace{1mu}  \omega_2^2\label{sff2}
\end{align}
\end{subequations}

Let  $\mathcal M_1$ and $\mathcal M_2$ be the matrices associated to the above quadratic forms:
\[
\mathcal M_1=\left[\begin{matrix}
a&b\\
b&c
\end{matrix}\right],\qquad
\mathcal M_2=\left[\begin{matrix}
e&f\\
f&g
\end{matrix}\right]
\]

The \emph{mean curvature} $\mathcal H$ is defined by:
\beq
\mathcal H=\dfrac{1}{2}\left(\mathcal H_1+\mathcal H_2\right),
\quad
\mathcal H_i=\hbox{Tr}\mspace{2mu}\mathcal M_i , \qquad i=1,2
\eeq
and it is easy to verify that similarly the Gaussian curvature is given by:
\beq
K=K_1+K_2 ,
\quad K_i= \det\mathcal M_i , \qquad i=1,2
\eeq

We can express the Gaussian, normal and mean curvature in terms of the coefficients of the second fundamental form \cite{little}:
\begin{align}\label{curvatures}
K=&(ac-b^2)+(eg-f^2)\\
\kappa=&(a-c)f-(e-g)b\notag\\
\mathcal H=&\dfrac{1}{2}(a+c)e_3+\dfrac{1}{2}(e+g)e_4\notag
\end{align}

\section{Monge form}

We consider a surface $S$ locally given by a parametrisation:
\[
\Xi : (x,y)\mapsto (x,y,\varphi(x,y),\psi(x,y))
\]
where $\Phi=(\varphi, \psi)$ has vanishing first jet at the origin,  $j^1\Phi(0)=0$.

The vectors $T_1$ and $T_2$ span the tangent space of $S$:
\[
T_1=\Xi_x=(1,0,\varphi_x,\psi_x),\quad T_2=\Xi_y=(0,1,\varphi_y,\psi_y)
\]
the index $z$ standing for derivative with respect to $z$.

The induced metric in $S$ is given by the first fundamental form:
\[
\fff=E\de x^2+2F\de x\de y+G\de y^2
\]
where:
\[
E=T_1\cdot T_1,\quad F=T_1\cdot T_2,\quad G=T_2\cdot T_2,
\]
We define:
\[
W=EG-F^2
\]

Instead of an orthonormal frame, it is more convenient to take a basis:
\begin{align}
T_1=&(1,0,\varphi_x,\psi_x),& & T_2=(0,1,\varphi_y,\psi_y)\\
\notag
N_1=&(-\varphi_x,-\varphi_y,1,0), && N_2=(-\psi_x,-\psi_y,0,1)
\end{align}
The vectors $T_1$ and $T_2$ span the tangent space, and the vectors $N_1$ and $N_2$ span the normal space. We define:
\[
\hat E=N_1\cdot N_1,\quad \hat F=N_1\cdot N_2,\quad \hat G=N_2\cdot N_2,
\]
and it is easy to verify that:
\[
\hat E  \hat G-\hat F^2=W
\]

Now consider the orthonormal frame defined by:
\begin{align}
e_1=&\dfrac{1}{\sqrt{E}}T_1,& & e_2=\sqrt{\dfrac{1}{EW}}\left(ET_2-F T_1\right)\\
\notag
e_3=&\dfrac{1}{\sqrt {\hat E}}N_1, && e_4=\sqrt{\dfrac{1}{\hat EW}}\left(\hat EN_2-\hat F N_1\right)
\end{align}
It is easy to see that:
\[
\omega_1(\dot x T_1+\dot y T_2)=\dfrac{1}{\sqrt{E}}(E\dot x+F\dot y), \quad
\omega_2(\dot x T_1+\dot y T_2)=\sqrt{\dfrac{W}{E}}\dot y
\]
or equivalently:
\beq\label{omegas}
\omega_1=\dfrac{1}{\sqrt{E}}(E\de x+F\de y), \quad
\omega_2=\sqrt{\dfrac{W}{E}}\de y
\eeq
Also:
\begin{align}\label{abc}
a=&\dfrac{1}{E\sqrt{\hat E}}\varphi_{xx}\\
\notag
b=&\dfrac{1}{E\sqrt{ W\hat E}}(E\varphi_{xy}-F\varphi_{xx})\\
\notag
c=&\dfrac{1}{EW\sqrt{\hat E}}(E^2\varphi_{yy}-2EF\varphi_{xy}+F^2\varphi_{xx})\\\label{efg}
e=&\dfrac{1}{E\sqrt{\hat E W}}\left(\hat E\psi_{xx}-\hat F\varphi_{xx}\right)\\
\notag
f=&\dfrac{1}{EW\sqrt{\hat E }}\left(E(\hat E\psi_{xy}-\hat F\varphi_{xy})-F(\hat E\psi_{xx}-\hat F\varphi_{xx})\right)\\
\notag
g=&\dfrac{1}{EW\sqrt{W\hat E }}\left(E^2(\hat E\psi_{yy}-\hat F\varphi_{yy})-\right.\\
\notag
&\phantom{\dfrac{1}{EW\sqrt{W\hat E }}}\left.-2EF(\hat E\psi_{xy}-\hat F\varphi_{xy})+F^2(\hat E\psi_{xx}-\hat F\varphi_{xx})\right)
\end{align}

Then, using these formul\ae\ or those from  \cite{am1,am2}, we obtain the following expressions for the Gaussian and normal curvature:

\begin{prop}
The Gaussian curvature is  given by:
\beq
K=\dfrac{1}{W^2}(\hat E H_\psi-\hat FQ+\hat G H_\varphi)\label{K}
\eeq
where:
\[
H_f=\He(f)=\left|
\begin{array}{cc}
f_{xx} &f_{xy}\\
f_{xy} & f_{yy}
\end{array}
\right|, \quad
Q=\left|
\begin{array}{cc}
\varphi_{xx} & \varphi_{xy}\\
\psi_{xy} & \psi_{yy}
\end{array}
\right|-\left|
\begin{array}{cc}
\varphi_{xy} & \varphi_{yy}\\
\psi_{xx} & \psi_{xy}
\end{array}\right|
\]
\end{prop}

\begin{prop}
The  normal curvature is given by:
\beq
\kappa=\dfrac{1}{W^2}(EL-FM+GN)\label{N}\
\eeq
where:
\[
L=\left|
\begin{array}{cc}
\varphi_{xy} & \varphi_{yy}\\
\psi_{xy} & \psi_{yy}
\end{array}
\right|, \quad
M=\left|
\begin{array}{cc}
\varphi_{xx} & \varphi_{yy}\\
\psi_{xx} & \psi_{yy}
\end{array}
\right|, \quad
N=\left|
\begin{array}{cc}
\varphi_{xx} & \varphi_{xy}\\
\psi_{xx} & \psi_{xy}
\end{array}
\right|
\]
\end{prop}

A surface immersed in $\mathbf R^4$ has an induced metric defined on it through the first fundamental form, and therefore an intrinsic Gauss curvature. Our previous definition of Gauss curvature agrees with it, and it is possible to prove more:

\begin{theorem}[Killing]
The intrinsic Gauss curvature $K_G$ of $S$ at a point $p\in S$ is the sum of the curvatures $K_1$ and $K_2$ of the projections $S_1$ and $S_2$ of the surface along any two orthogonal normal directions $n_2\in N_p S$ and $n_1\in N_p S$ respectively.
\end{theorem}
\begin{proof}By a linear change of coordinates and a translation of the origin, we can assume that $n_1\in N_p S$ spans the third axis and $n_2\in N_p S$ the fourth, and also that $p$ is the origin.

The surface $S$ is  locally given by a parametrisation:
\[
\Xi : (x,y)\mapsto (x,y,\varphi(x,y),\psi(x,y))
\]
where $\Phi=(\varphi, \psi)$ has vanishing first jet at the origin,  $j^1\Phi(0)=0$.

The intrinsic Gauss curvature $K_G$ of $S$ is given by Brioschi formula \cite{spivak2}:
\beq
K_G=\dfrac{
\left|\begin{matrix}
-\dfrac{1}{2}E_{yy}+F_{xy}-\dfrac{1}{2}G_{xx}&\dfrac{1}{2}E_x&F_x-\dfrac{1}{2}E_y\\[8pt]
F_y-\dfrac{1}{2}G_x&E&F\\[8pt]
\dfrac{ 1}{2}G_y&F&G
\end{matrix}\right |
-
\left|\begin{matrix}
0&\dfrac{1}{2}E_y&\dfrac{1}{2}G_x\\[8pt]
\dfrac{1}{2}E_y&E&F\\[8pt]
\dfrac{1}{2}G_x&F&G
\end{matrix}\right|}
{(EG-F^2)^2}
\eeq

At the origin:
\[
E=G=1, \quad F=0
\]
and all first order derivatives of $E$, $F$ and $G$ vanish. Thus the Brioschi formula gives:
\[
K_G=-\dfrac{1}{2}E_{yy}+F_{xy}-\dfrac{1}{2}G_{xx}\quad\hbox{at the origin}
\]
The surfaces $S_1$ and $S_2$ are the graphs of $\varphi$ and $\psi$ respectively, and their intrinsic Gauss curvatures agree with the definition of $K_1=ac-b^2$ and $K_2=eg-f^2$ above.

If $E_i$, $F_i$ and $G_i$ are the coefficients of the first fundamental forms of $S_i$, $i=1,2$, we have:
\[
E=E_1+E_2-1,\quad F=F_1+F_2,\quad G=G_1+G_2-1
\]
and therefore it follows from linearity of the derivatives that:
\[
K_G=K_1+K_2
\]
\end{proof}

As we have remarked before, the Gauss curvature can be given by $K=K_1+K_2$, and therefore it agrees with the intrinsic Gauss curvature $K_G$.

\section{Curvature ellipse}

The \emph{curvature ellipse} or \emph{indicatrix} $\mathcal E$ of the surface $S$ is the image under the second fundamental form of the unit circle in the tangent space:
\[
\mathcal E_p=\{v\in N_pS \ | \ v=\sff(u), \ u\in T_pS, \ |u|=1\}
\]

Let $u\in T_pS$ with $|u|=1$; then $\sff (u)$ is the normal curvature vector at $p$ of any curve $\gamma$ on $S$ such that:
\[
\gamma(0)=p,\quad \dot\gamma(0)=u
\]
and in fact it is the curvature if we choose $\gamma$ appropriately:

\begin{lemma}
Let $\gamma$ be the curve passing though $p$, parametrized by arc length from $p$, obtained as the intersection of the surface $S$ with the hyperplane containing the normal space at $p$ and $u$. Then, if $\chi_{\gamma}$ is the curvature of $\gamma$ at $p$:
\[
\chi_{\gamma}=\sff(u)
\]
\end{lemma}
\begin{proof}
As $\gamma$ is a plane curve parametrized by arc length we have:
\[
\dfrac{d^2}{ds^2}\gamma(0)=\chi_{\gamma}, \qquad \dfrac{d^2}{ds^2}\gamma(0)\perp u=\dfrac{d}{ds}\gamma(0)
\]
and therefore the second derivative has only normal component and it is given by the second fundamental form $\sff(u)$.
\end{proof}

As  $u$ describes the unit circle in the tangent space, its image $\sff (u)$ describes the curvature ellipse:

\begin{theorem}[Moore, Wilson \cite{mw}]
The indicatrix $\mathcal E$ of the surface $S$ is an ellipse.
\end{theorem}
\begin{proof}
Consider the map $\eta$ from the unitary tangent bundle $UTS$ of $S$ into the normal bundle $NS$ given by:
\beq
\eta (\theta)=\sff (\cos\theta\, e_1+\sin\theta\, e_2)
\eeq

From its definition:
\begin{align}
\eta (\theta)=&(a\cos^2\theta+2b\cos\theta\sin\theta+c\sin^2\theta)e_3+\label{10}\\
&+(e\cos^2\theta+2f\cos\theta\sin\theta+g\sin^2\theta)e_4 \notag\\
=&\mathcal H+\dfrac{1}{2}((a-c)\cos2\theta+b\sin2\theta)e_3+\notag\\
&+\dfrac{1}{2}((e-g)\cos2\theta+f\sin2\theta)e_4\notag
\end{align}

Let:
\[
\mathcal A=\left[
\begin{matrix}
\dfrac{1}{2}((a-c)&b\\[10pt]
\dfrac{1}{2}(e-g)&f\end{matrix}
\right]
\]
Then:
\[
\eta (\theta)=\mathcal H+\mathcal Aw, \quad w=(\cos2\theta,\sin2\theta)
\]
Thus the indicatrix, the image of $\eta$, is an ellipse, possibly singular, centred ar $\mathcal H$.
\end{proof}

\begin{prop}[\cite{little}]
The normal curvature $\kappa$ is related to the oriented area $A$ of the curvature ellipse by:
\beq
\dfrac{\pi}{2}\kappa=A
\eeq
\end{prop}
\begin{proof}
As:
\[
\eta (\theta)=\mathcal H+\mathcal Aw, \quad w=(\cos2\theta,\sin2\theta)
\]
$\eta (\theta)$ describes twice an ellipse, the curvature ellipse or indicatrix, centred at $\mathcal H$; the oriented area of the ellipse will then be the area of the unit circle multiplied by the determinant of the matrix $\mathcal A$
and therefore:
\[
A=\pi\dfrac{1}{2}((a-c)f-(e-g)b)=\dfrac{\pi}{2}\kappa
\]
\end{proof}

The curvature ellipse at a point  $p\in S$ can be used to  characterize that point; in particular:
\begin{itemize}
\item  $p$ is a \emph{circle point} if the curvature ellipse at $p$ is a circumference.
\item $p$ is a \emph{minimal point} if the curvature ellipse at $p$ is centred at the origin, $\mathcal H(p)=0$.
\item $p$ is an \emph{umbilic point} if the curvature ellipse at $p$ is a circumference centred at the origin; the point is both a minimal and a circle point.
\end{itemize}

At a non umbilic point it is always possible to find canonical moving frames~\cite{w} for which the computations are easier:

\begin{prop}\label{canframe}
Given any point $p\in S\subset \Rq$ such that $p$ is not an umbilic point, there exists  a canonical moving frame around $p$ for which:
\begin{itemize}
\item $b\equiv 0$.
\item $e\equiv g$
\item $\dfrac{1}{2}(a-c)\ge |f |\ge 0$
\end{itemize}
\end{prop}

\begin{proof}
We choose $e_3\in N_pS$ parallel to the major axis of the ellipse of curvature, and $e_4\in N_pS$ normal to it; then $e_1\in T_pS$ is chosen along the direction whose image under the second fundamental form is spanned by $e_3$, and $e_2\in T_pS$ normal to $e_1$, so that $\{e_1,e_2\}$ has the correct orientation. If the ellipse of curvature is a circle (not centred at the origin) the direction of $e_3$ is the line defined by the origin and the centre of that circle; if the ellipse degenerates into a radial segment, $e_3$ is chosen along the line spanned by the segment.

The ambiguity in the choices of $e_1$ and $e_3$ allows  $\{e_1, e_2, e_3, e_4\}$ to have the standard orientation in $\Rq$, and also to have $a-c\ge 0$.
\end{proof}

Now, $(a-c)/2$ and $|f|$ are the major and minor semi-axes of the curvature ellipse respectively, and formul\ae\ ~(\ref{curvatures}) become:
\begin{align}\label{cancurv}
K=&ac+e^2-f^2\\
\kappa=&(a-c)f\notag\\
\mathcal H=&\dfrac{1}{2}(a+c)e_3+e\, e_4\notag
\end{align}

A necessary and sufficient condition~\cite{little} for $p$ to be a circle point is that:
\beq
\mathcal H^2-K=|\kappa|
\eeq
In fact, we have:

\begin{wineq}[\cite{wintgen}]
If $S$ is an immersed  surface in $\Rq$, then at every point $p\in S$ we have the inequality:
\beq
\label{wineq}
\mathcal H^2\ge K+|\kappa |
\eeq
The point $p$ is a circle point if and only if $\mathcal H^2 = K+|\kappa |$.
\end{wineq}
\begin{proof}
Using a canonical moving frame around $p$, assumed to be not an umbilic point, we have:
\begin{align}
0 &  \le (a-c-2|f|)^2=(a-c)^2+4f^2-4(a-c)|f|=\label{14}\\
\notag &=(a-c)^2+4f^2-4|\kappa|=a^2+c^2-2ac+4f^2-4|\kappa|=\\
\notag &=a^2+c^2+2f^2+2e^2-2K-4|\kappa|\\
4\mathcal H^2 & = (a+c)^2+4e^2=a^2+c^2+2ac+4e^2\label{15}\\
\notag &=a^2+c^2+2f^2+2e^2+2K
\end{align}
The inequality follows immediately from (\ref{15})-(\ref{14}):
$
4\mathcal H^2\ge 4K+4|\kappa|
$

The curvature ellipse is a circle if and only if the two semi-axes are equal:
\[
\dfrac{1}{2}(a-c)=|f|
\]
and this is exactly when we have equality above.

There remains to consider the case where $p$ is an umbilic point, where we should have $K+|\kappa|\equiv 0$; but at an umbilic point we must have $a+c=0$ and $e+g=0$ and:
\[
\mathcal A=\left[
\begin{matrix}
\dfrac{1}{2}((a-c)&b\\[10pt]
\dfrac{1}{2}(e-g)&f\end{matrix}
\right]=
\left[
\begin{matrix}
a&b\\
&\\
e&f\end{matrix}
\right]
\]
a multiple of an orthogonal matrix, so:
\[
a^2+b^2=e^2+f^2=R^2, \quad (a,b)\perp (e,f)
\]
It follows that  $|a|=|f|$, $ |b|=|e| $ and $K=-|\kappa|$ as desired.
\end{proof}

Thus at an umbilic point we always have a nonpositive Gaussian curvature.

By identifying $p$ with the origin of $N_p S$,  the points of $S$ may be classified according to their position with respect to the curvature ellipse, that we assume to be non degenerate ($\kappa(p)\ne 0$), as follows:

\begin{itemize}
\item  $p$ lies outside the curvature ellipse.

\noindent
The point is said to be a
\emph{hyperbolic point} of $S$. The \emph{asymptotic directions} are the tangent directions whose images span the two normal lines tangent to the indicatrix passing through the origin;   the \emph{binormals} are the normal directions perpendicular  to those normal lines.

 \begin{figure}[thhh]
\begin{center}
\includegraphics[width= \linewidth]{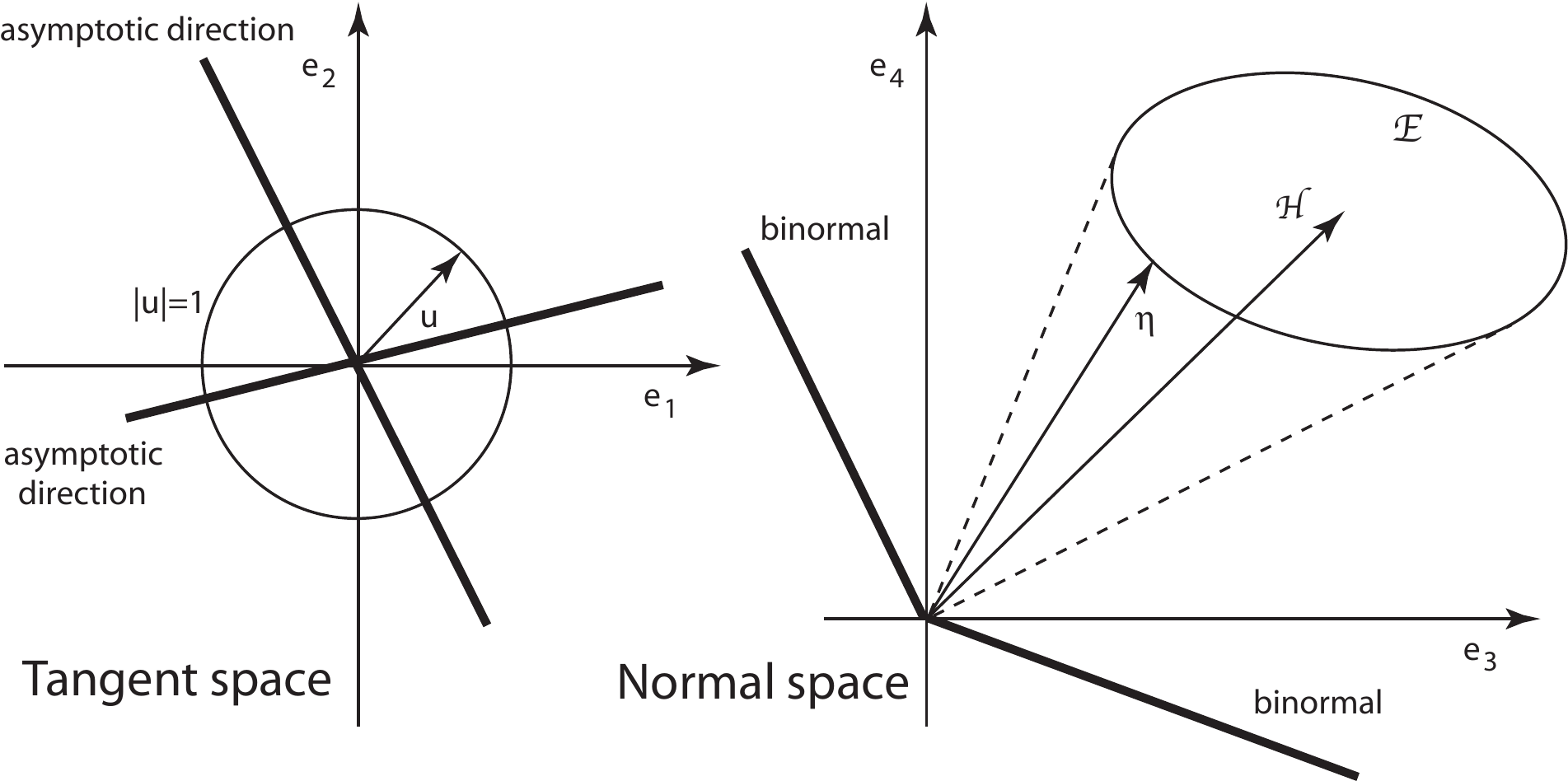}
\end{center}
\caption{Indicatrix at a hyperbolic point: 2 binormals}\label{indxas}
\end{figure}

\item $p$ lies inside the curvature ellipse.

\noindent
The point $p$  is an \emph{elliptic point}. There are no binormals and no asymptotic directions.

\begin{figure}[hhhh]
\begin{center}
\includegraphics[width=0.6 \linewidth]{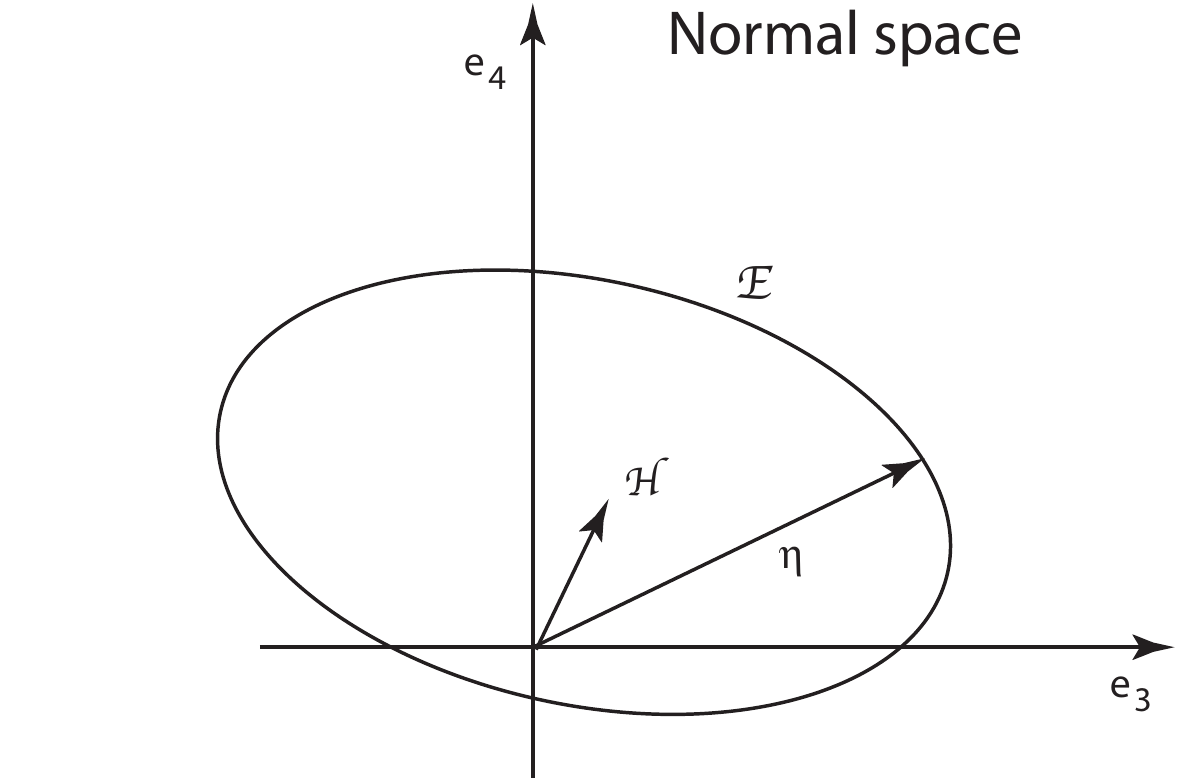}
\end{center}
\caption{Indicatrix at an elliptic point: no binormals}\label{indxe}
\end{figure}

\item $p$ lies on the curvature ellipse.

\noindent
The point $p$ is a \emph{parabolic point}. There is  one binormal and one asymptotic direction.

\begin{figure}[hhhh]
\begin{center}
\includegraphics[width=0.9 \linewidth]{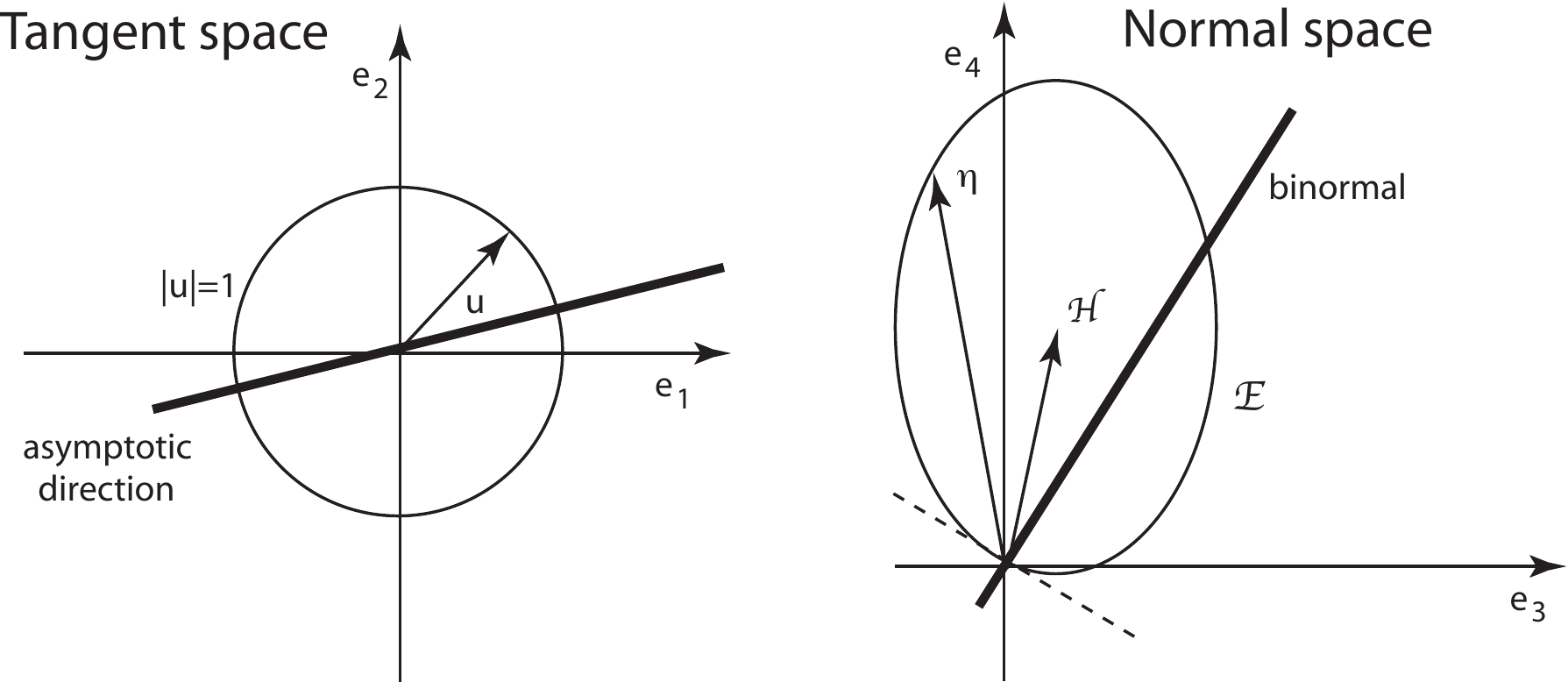}
\end{center}
\caption{Indicatrix at a parabolic point: 1 binormal}\label{indxp}
\end{figure}

\end{itemize}

The points $p$ where the curvature ellipse passes through the origin are characterised by $\Delta(p)=0$, where:
\beq
\Delta=\dfrac{1}{4}\left|
\begin{matrix}
a&2b&c&0\\
e&2f&g&0\\
0&a&2b&c\\
0&e&2f&g
\end{matrix}
\right|
\eeq
In fact, $\Delta$ is the resultant of the two polynomials $ax^2+2bxy+cy^2$ and $ex^2+2fxy+gy^2$. If $\eta(\theta)=0$ those polynomials have a common root $(\cos\theta,\sin\theta)$, and their resultant has to be zero.

The points of $S$ may be classified~\cite{little} using $\Delta$, as follows:

\begin{prop}
If a pont $p\in S$ is hyperbolic, parabolic or elliptic then  $\Delta(p) < 0$, $\Delta(p)=0$ or $\Delta(p) > 0$, respectively.
\end{prop}

\begin{proof}
Let $\tau\in T_pS$ be a tangent vector and $\mathfrak  N$ be  defined by:
\[
\de \tau\cdot e_3\wedge \de \tau\cdot e_4=\mathfrak  N\ \omega_1\wedge \omega_2
\]
As before, the exterior derivative is taken componentwise.

It $\tau =x e_1+ y e_2$, or $\omega_1(\tau)=x$ and $\omega_2(\tau)=y$, then $\mathfrak  N$ is a quadratic form on $(x,y)$ given by:
\[
\mathfrak  N(x,y)=(af-be)x^2+(ag-ce)xy+(bg-cf)y^2
\]

A straightforward computation gives:
\begin{align}\label{Delta}
\Delta=&(ac-b^2)(eg-f^2)-\dfrac{1}{4}(ag+ce-2bf)^2=\\
\notag
=&(af-be)(bg-cf)-\dfrac{1}{4}(ag-ce)^2=
\det \mathfrak  N
\end{align}
and therefore the equation $\mathfrak  N(x,y)=0$ on the direction defined by $(x,y)$ has two solutions, one or no solutions as $\Delta(p) < 0$, $\Delta(p)=0$ or $\Delta(p) > 0$, respectively. 

On the other hand, $\mathfrak  N(x,y)=0$ is equivalent to:
\[
\de \tau\cdot e_3\wedge \de \tau\cdot e_4 \ (u,v)=0,\qquad \forall u,v\in T_pS
\]

If we define:
\[
\de_N\tau=(\de \tau\cdot e_3) e_3+(\de \tau\cdot e_4) e_4
\]
we see that:
\[
\de \tau\cdot e_3\wedge \de \tau\cdot e_4 \ (u,v)=
\left|
\begin{matrix}
\de \tau(u)\cdot e_3&\de \tau(u)\cdot e_4\\
\de \tau(v)\cdot e_3&\de \tau(v)\cdot e_4
\end{matrix}
\right|=\de_N\tau(u)\wedge \de_N\tau(v)
\]
and thus $\mathfrak  N(x,y)=0$ is equivalent to the image of $\de_N\tau$ being one dimensional. When that happens, the image of  $\de_N\tau$ spans  a line tangent to the curvature ellipse and perpendicular to a binormal:

Let:
\[
\tau=\rho u, \quad u=\dfrac{1}{|\tau|}\tau=\cos\theta e_1+\sin\theta e_2, \quad \rho=|\tau|
\]
and $\gamma$ be  the curve passing though $p$, parametrized by arc length from $p$, obtained as the intersection of the surface $S$ with the hyperplane containing the normal space at $p$ and $u$. Then, taking the normal component of  $\de \tau(u)$:
\[
\de_N\tau(u)=\rho\,\sff(u)=\rho\,\eta(\theta)
\]
Taking $v=-\sin\theta e_1+\cos\theta e_2$, and since $\de_N\tau(v)\parallel \de_N\tau(u)$:
\[
\eta'(\theta)=\dfrac{1}{\rho}\de_N\tau(v)\parallel \eta(\theta)
\]
This means that the tangent to the indicatrix at $\eta(\theta)$ passes through the origin, as $\eta'(\theta)\parallel \eta(\theta)$, and therefore the image of $\de_N\tau$ spans a line tangent to the curvature ellipse and perpendicular to a binormal.
\end{proof}

We can  extend the definition of hyperbolic point, respectively elliptic point and parabolic point, to the case where $\kappa(p)=0$ by means of $\Delta$, as $\Delta(p)<0$, respectively $\Delta(p)>0$ and $\Delta(p)=0$.

\begin{defn}
The second-order \emph{osculating space} of the surface $S$ at $p\in S$ is the space generated by all vectors $\gamma'(0)$ and $\gamma''(0)$ where $\gamma$ is a curve through $p$ parametrized by arc length from $p$. An \emph{inflection point} is a point where the dimension of the osculating space is not maximal.
\end{defn}

\begin{theorem}The following conditions are equivalent:
\begin{itemize}
\item $p\in S$ is an inflection point.
\item $p\in S$ is a point of intersection of $\Delta=0$ and $\kappa=0$.
\item ${\rank}\mathcal M(p) \le 1$
\end{itemize}
The inflection points are  singular points of $\Delta=0$.
\end{theorem}

\begin{proof}
Assume $p\in S$ is a point of intersection, $\Delta(p)=0$ and $\kappa(p)=0$. Since:
\[
\Delta=(af-be)(bg-cf)-\dfrac{1}{4}(ag-ce)^2
\]
it follows from $\Delta =0$ that $(af-be)(bg-cf)\ge 0$; now as
\[
\kappa=(a-c)f-(e-g)b=(af-be)+(bg-cf)
\]
we see that, at $p$:
\[
af-be=0,\quad bg-cf=0,\quad\hbox{and from }\Delta(p)=0, \quad ag-ce=0
\]
This implies ${\rank}\mathcal M(p) \le 1$ and therefore the indicatrix is a radial segment and the point is an inflection point: the osculating space is three dimensional, while it is four dimensional when the curvature ellipse is non degenerate. 
The vanishing of those three expression  also implies that the derivatives of $\Delta$ are both zero at $p$, which is then a singular point.

We can reverse the argument to show that at an inflection point $p$ we must have $\Delta(p)=0$ and $\kappa(p)=0$.

\end{proof}

\begin{prop}
Let $p\in S$ be a generic inflection point. Then $p$ is a Morse singular point of $\Delta=0$,  and the Hessian $H_{\Delta}$ of $\Delta$ at $p$ has the same sign as the curvature $K(p)$.
\end{prop}

\begin{proof}
Since $p$ is generic we can assume that $K(p)\ne 0$. By a linear change of coordinates  and a translation of the origin in $\Rq$, if necessary, we can assume that $p$ is the origin, the tangent plane $T_pS$ is  the $(x,y)$ plane, and  the line passing through the origin containing the curvature ellipse is the third axis.

We consider $S$ around $p$ as the graph of a map $(\varphi,\psi)$ around the origin.in $\mathbf R^2$ Then:
\[
\varphi(0)=\psi(0)=0,\quad \varphi_x(0)=\psi_x(0)=0,\quad \varphi_y(0)=\psi_y(0)=0
\]

The condition of $p$ being an inflection point, and the choice of the third axis mean that, in view of  (\ref{efg}):
\[
\psi_{xx}(0)=e(0)=0, \quad \psi_{xy}(0)=f(0)=0, \quad \psi_{yy}(0)=g(0)=0
\]
and $\psi$ is a homogeneous cubic polynomial plus higher order terms. Note that at the origin $E=G=\hat E=\hat G=1$ and $F=\hat F=0$.

A convenient standard change of coordinates $(x,y)$ and the genericity condition allow us to assume that:
\[
\psi(x,y)=\dfrac{1}{3}x^3+\mu xy^2+O(4), \quad \mu\ne 0
\]

Thus:
\[
e=2x+O(2),\quad f=2\mu y+O(2),\quad g=2\mu x+O(2)
\]
and $a=A+O(1)$, $b=B+O(1)$, $c=C+O(1)$. Then:
\begin{align*}
\Delta =&(ac-b^2)(eg-f^2)-\dfrac{1}{4}(ag+ce-2bf)^2=\\
=&(AC-B)^2)(4\mu x^2-4\mu^2y^2)-\\
&-\dfrac{1}{4}(2A\mu x+2Cx-4B\mu y)^2+O(3)\\
=&-[4\mu B^2+(C-\mu A)^2]x^2+4\mu(C+\mu A)Bxy-4\mu^2ACy^2+O(3)
\end{align*}
and:
\begin{align*}
H_{\Delta} =&
\left |
\begin{matrix}
-8\mu B^2-2(C-\mu A)^2 & 4\mu (C+\mu A)B\\
&\\
4\mu (C+\mu A)B & -8\mu^2AC
\end{matrix}
\right |=\\
=&16 \mu^2(C-\mu A)^2AC+(4\mu AC)(16\mu^2B^2)-16\mu^2(C+\mu A)^2B^2=\\
=&16 \mu^2(C-\mu A)^2[AC-B^2]
\end{align*}

As $K(0)=a(0)c(0)-b(0)^2+e(0)g(0)-f(0)^2=AC-B^2$ we finally obtain:
\[
H_{\Delta} =16 \mu^2(C-\mu A)^2 K(0)
\]
with $C-\mu A\ne 0$  by genericity again.
\end{proof}

When $\Delta(p)=0$ we can distinguish among the following possibilities:

\begin{itemize}
\item $\Delta(p)=0$, $K(p) < 0 $ and ${\rank}\mathcal M(p) = 2$

\noindent
The curvature  ellipse is non-degenerate, $\kappa(p)\ne 0$; the binormal is the normal  at $p$.

\item $\Delta(p) = 0$, $K(p) < 0 $ and ${\rank}\mathcal M(p) = 1$

\noindent
$p$ is an \emph{inflection point
of real type}: the curvature ellipse is a radial segment and $p$ does not belong to it, $\kappa(p)= 0$. The point $p$ is a self-intersection point of $\Delta=0$, as $H_{\Delta}(p)<0$.
\item $\Delta(p) = 0$, $K(p)=0$

\noindent
$p$ is an \emph{inflection point of flat type}: the curvature ellipse is a radial segment and $p$  belongs to its boundary, $\kappa(p)= 0$.
\item $\Delta(p) = 0$, $K(p) > 0$

\noindent
$p$ is an \emph{inflection point of imaginary type}: the curvature ellipse is a radial segment and $p$ belongs to its interior, $\kappa(p)= 0$. The point $p$ is an isolated point of $\Delta=0$, as $H_{\Delta}(p)>0$.

\end{itemize}

\begin{figure}[thhh]
\begin{center}
\includegraphics[width= \linewidth]{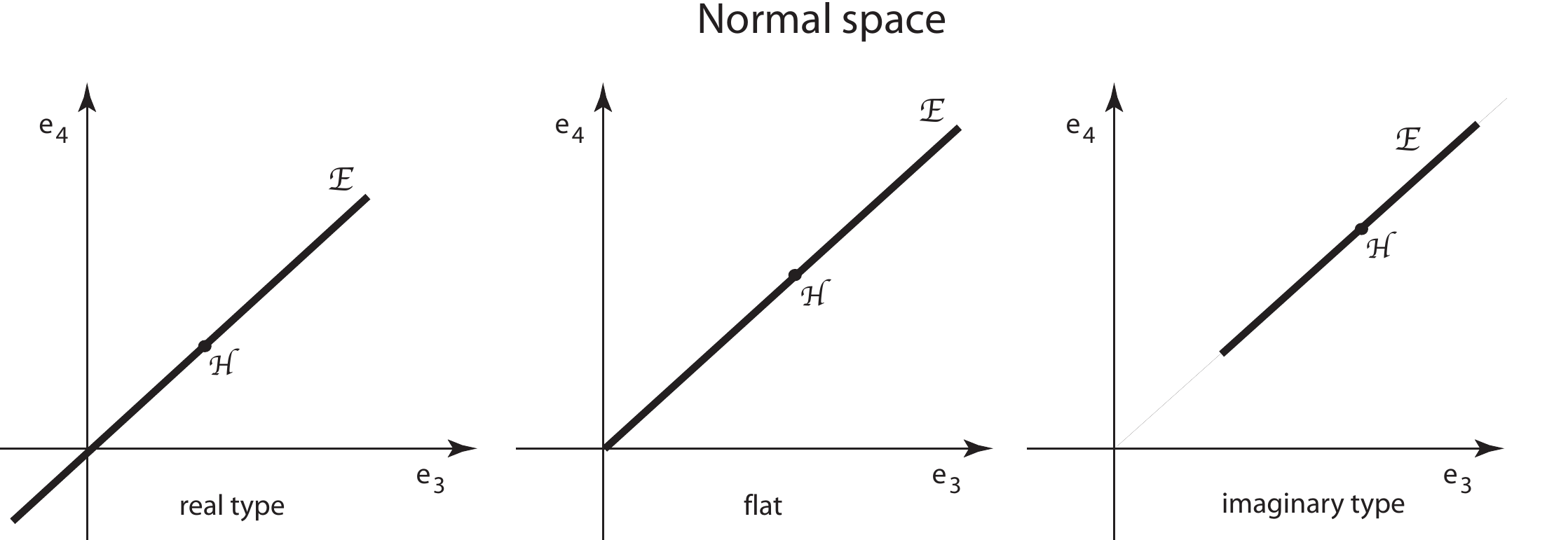}
\end{center}
\caption{Inflection points}\label{inflex}
\end{figure}

At an inflection point the normal to the line through the origin containing the radial segment defines the binormal.

\begin{rem}
If $\Delta(p) = 0$ and $K(p)\ge 0$ then ${\rank}\mathcal M(p)\le 1$: if $K(p)\ge 0$ then $K_1(p)$ and $K_2(p)$ cannot be both positive, therefore $\sff_1$ or $\sff_2$ has a double real root or no real roots; $\Delta(p) = 0$ forces all roots to be the same, or all non real, and $\sff_1$ and $\sff_2$ to be multiples (all roots are common), thus ${\rank}\mathcal M(p)\le 1$.
\end{rem}

\begin{rem}\label{flat}
It can be shown that for an open and dense set of embeddings of $S$ in $\mathbf R^4$, $\Delta^{-1}(0) \cup
K^{-1}(0) = \emptyset$; therefore on a generic surface there are no inflection points of flat type.
\end{rem}

\begin{figure}[hhhh]
\begin{center}
\includegraphics[width=0.9 \linewidth]{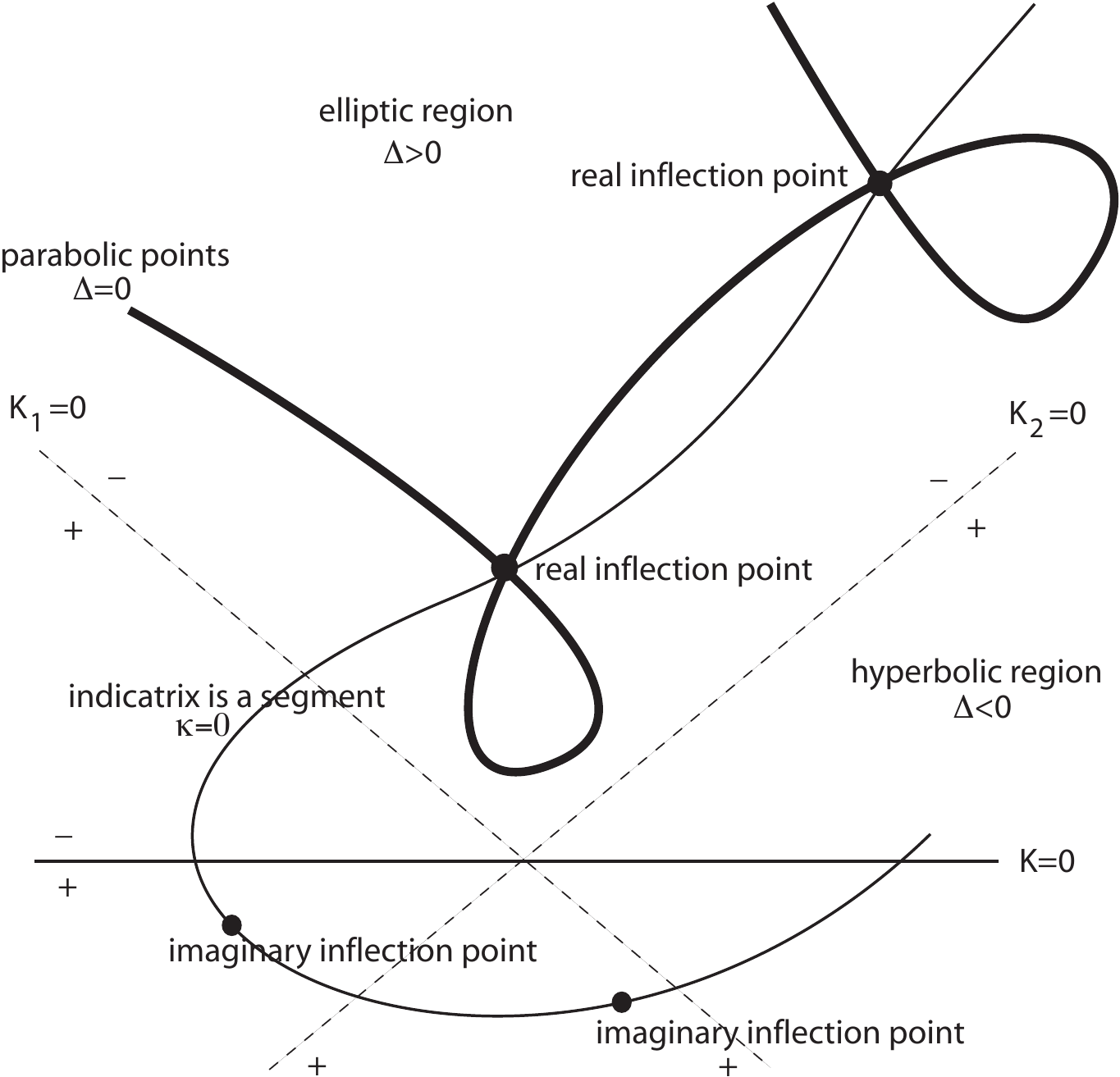}
\end{center}
\caption{Generic inflection points}\label{gendiagram}
\end{figure}

\newpage
\section{The characteristic curve}

Let $\gamma$ be a curve passing though $p$,  and   $u$ be the  tangent vector defined along $\gamma$ by:
\[
u(s)=\dfrac{d}{ds}\gamma(s)
\]
Consider a vector field $w$ along the curve $\gamma$; taking its derivative with respect to $s$, the resulting vector is not necessarily tangent to the surface. The tangent component of that derivative is the \emph{covariant derivative} of $w$ along $u$, denoted $\nabla_u w$.

Assume $\gamma$  to be a curve passing though $p$, parametrized by arc length from $p$, then   $u$, constructed as above, will be a unit tangent vector.
We define another unit  tangent vector field  $v$ along $\gamma$ so that $u$ and $v$ form an  orthonormal basis of the tangent space with the positive orientation.

\begin{prop}
Let $u$ and $v$ be  orthogonal unit tangent vectors along  a curve $\gamma$ parametrized by arc length from $p$, so that $\{u,v\}$ form a basis of the tangent space with the positive orientation. Then:
\[
\dfrac{d}{ds}u=\eta +\alpha v, \quad \dfrac{d}{ds}v=\zeta -\alpha u
\]
where the normal components are $\eta= \sff(u)$ and $\zeta$, and the tangent components are:
\[
\nabla_uu=\alpha v, \quad \nabla_uv=-\alpha u
\]
Moreover $\xi=\eta-\mathcal H$ and $\zeta$ are conjugate radii of the curvature ellipse.
\end{prop}

\begin{figure}[thhh]
\begin{center}
\includegraphics[width=\linewidth]{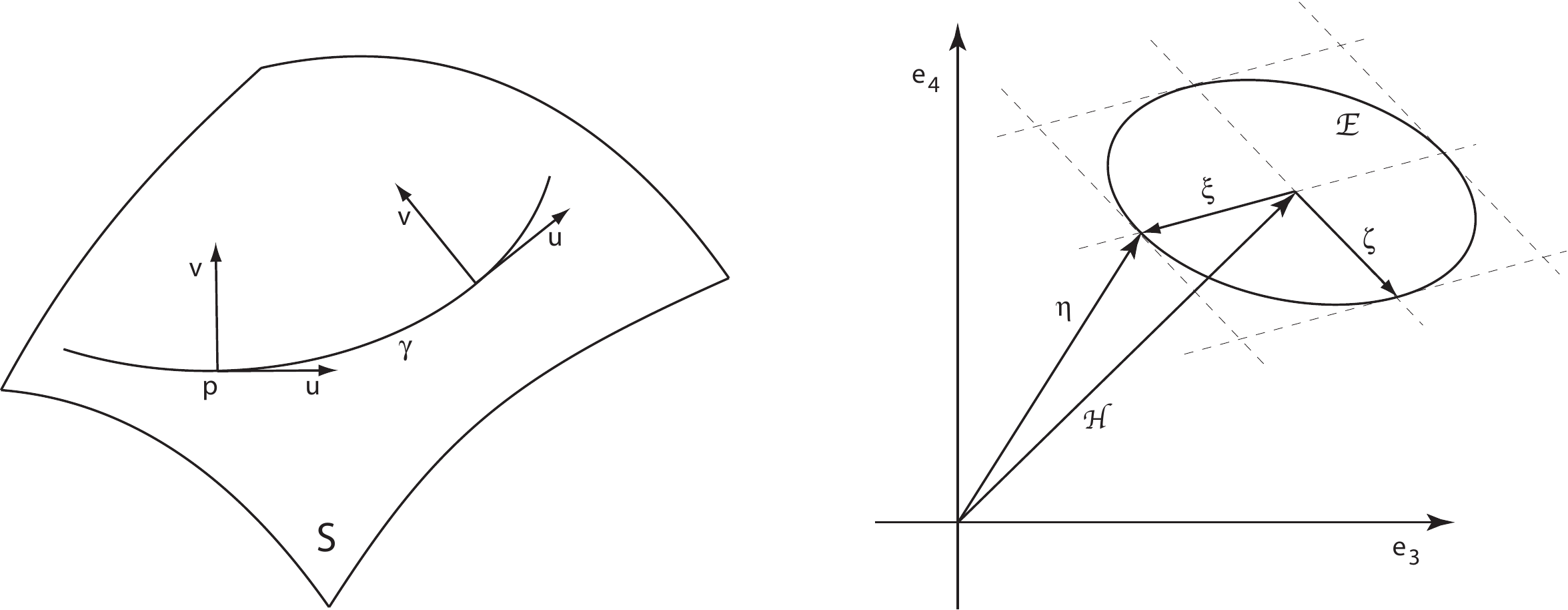}
\end{center}
\caption{$\xi$ and $\zeta$ are conjugate radii}\label{indx}
\end{figure}

\begin{rem}
The relations  concerning the tangent components of the above derivatives also follow by taking the covariant derivative $\nabla_u$ of $u\cdot u=1$, $v\cdot v=1$ and $u\cdot v=0$.
\end{rem}

\begin{rem} If $\gamma$ is a geodesic then $\nabla_uu \equiv 0$ and therefore $\alpha\equiv 0$; also if  $\gamma$ is in the normal section at $p$, the intersection of the surface with the hyperplane containing the normal space at $p$ and $u(0)$, then $\alpha(0)=0$.
\end{rem}

\begin{proof}
Take:
\[
u=\cos \theta(s)e_1+\sin \theta(s)e_2, \quad v=- \sin\theta(s)e_1+ \cos\theta(s)e_2
\]
Then:
\begin{align*}
\dfrac{d}{ds}u &=\cos\theta  (\omega _{12}e_2+\omega _{13}e_3+\omega _{14}e_4) +\sin\theta  (\omega _{21}e_1+\omega _{23}e_3+\omega _{24}e_4)+\\ 
 &\phantom{salto} +\dot\theta (- \sin\theta e_1+ \cos\theta e_2)=\\
&= \eta+(\omega _{12}+\dot\theta)v,  \qquad \hbox{ \ where \ } \eta=\sff(u)
\end{align*}

Similarly:
\begin{align*}
\dfrac{d}{ds}v &=-\sin\theta  (\omega _{12}e_2+\omega _{13}e_3+\omega _{14}e_4) +\cos\theta  (\omega _{21}e_1+\omega _{23}e_3+\omega _{24}e_4)+\\ 
 &\phantom{salto} +\dot\theta (- cos\theta e_1- \sin\theta e_2)=\\
&=-\sin\theta  (\omega _{13}e_3+\omega _{14}e_4)+\cos\theta  (\omega _{23}e_3+\omega _{24}e_4) -(\omega _{12}+\dot\theta)u=\\
&=[(c-a)\sin\theta\cos\theta+b\cos 2 \theta]e_3+\\
&\phantom{salto} +[(g-e)\sin\theta\cos\theta+f\cos2\theta] e_4-\alpha u
\end{align*}
where $ \alpha=\omega _{12}+\dot\theta$.

Now:
\begin{align*}
\zeta &=[(c-a)\sin\theta\cos\theta+b\cos 2 \theta]e_3+[(g-e)\sin\theta\cos\theta+f\cos2\theta] e_4=\\ 
 &=  \mathcal A (-\sin 2\theta , \cos 2\theta)
\end{align*}

Since:
\[
\xi=\eta-\mathcal H=\mathcal A (\cos 2\theta, \sin 2\theta )
\]
we see that  $\xi$ and $\zeta$ are conjugate radii of the  ellipse $\mathcal A(S^1)$, being the images of two perpendicular radii. Considered as applied at the end point of $\mathcal H$, they are conjugate radii of the curvature ellipse.
\end{proof}

We will use the bivector given by the wedge product $v_1\wedge v_2$ to denote the two plane $P$ defined by 
the oriented pair of linearly independent vectors $v_1$ and $v_2$. The bivector $v_1\wedge v_2$ represents an oriented area, that of the oriented parallelogram defined by the vectors $v_1$ and $v_2$; if we choose different independent vectors $v_1'$ and $v_2'$ in the same plane, and with the same orientaten, their wedge product is $v_1'\wedge v_2'=\lambda(v_1\wedge v_2)$ with $\lambda>0$. Thus the (oriented) line spanned by $v_1\wedge v_2$ charcterizes the (oriented)  plane defined by the vectors $v_1$ and $v_2$.

The inner product of a vector $u$ and a bivector $P=v_1\wedge v_2$ is defined as:
\[
u\bullet P=(u\cdot v_1)v_2-(u\cdot v_2)v_1
\]
Therefore $u\bullet P$ is a vector in $P$ orthogonal to the projection $\pi_P(u)$ of $u$ on the plane $P$, and $\{\pi_P(u), u\bullet P \}$ has the same orientation as $\{v_1, v_2 \}$; also $u\bullet P=0$ is equivalent to $u\perp P$.

\begin{lemma}
Let $N_s$ be the family of normal spaces along $\gamma$; the evolvent  of that family at $s=0$ is given by $n\in N_pS$ such that:
\[
n\cdot \xi=1,\quad n\cdot \zeta=0
\]
where $\xi=\sff(u) $ and $\zeta$ is the conjugate radius of $\eta=\xi-\mathcal H$
\end{lemma}

\begin{proof}
The equation for $N_s$ is:
\[
(w-\gamma(s))\bullet u(\gamma(s))\wedge v(\gamma(s))\equiv 0
\]
and therefore the evolvent at $s=0$ is defined by:
\[
n\bullet u\wedge v= 0, \quad \dfrac{d}{ds}\left[(w-\gamma(s))\bullet u(\gamma(s))\wedge v(\gamma(s))\right]_{s=0}=0
\]
where $n=w-p$.
Now:
\begin{align*}
 \dfrac{d}{ds}[(w-\gamma(s)) & \bullet u(\gamma(s))\wedge v(\gamma(s))]= \\
 &=\dfrac{d}{ds}(w-\gamma(s))\bullet (u(\gamma(s))\wedge v(\gamma(s)))+\\
 &+(w-\gamma(s))\bullet 
 \dfrac{d}{ds} u(\gamma(s))\wedge v(\gamma(s)) +\\
 &+(w-\gamma(s))\bullet u(\gamma(s))\wedge \dfrac{d}{ds}v(\gamma(s))
\end{align*}

Since:
\[
\dfrac{d}{ds}\gamma(0)=u
\]
we have:
\[
\left[\dfrac{d}{ds}(w-\gamma(s))\bullet u(\gamma(s))\wedge v(\gamma(s))\right]_{s=0}=
-u\bullet u\wedge v=-v
\]

Also:
\[
 \dfrac{d}{ds} u(\gamma(0))=\sff(u), \quad  \dfrac{d}{ds} v(\gamma(0))=\zeta
\]
and thus:
\[
\dfrac{d}{ds}\left[(w-\gamma(s))\bullet u(\gamma(s))\wedge v(\gamma(s))\right]_{s=0}=
-v+n\bullet \sff(u)\wedge v+n\bullet u\wedge \zeta
\]

The condition:
\[
-v+n\bullet \sff(u)\wedge v+n\bullet u\wedge \zeta=-v+\left(n\cdot \sff(u)\right)v-\left(n\cdot \zeta\right)u=0
\]
is equivalent to:
\[
n\cdot \sff(u)=1, \quad n\cdot \zeta=0
\]
\end{proof}

The normal vector $n$ is the  {\it intersection of consecutive normal planes} along the direction $u$: let $\bar n$ be the normal vector such that $\bar n\perp T_pS$ and $\bar n\perp T_{\gamma(\bar s)}S$, or equivalently $\bar n\in N_p\cap N_{\gamma(\bar s)}S$; then $n=\lim_{\bar s\rightarrow 0}\bar n$.

\begin{defn}
The \emph{characteristic curve} $\mathcal C$ is the curve on the normal space $N_pS$ described by the normal vector $n$, the intersection of consecutive normal planes, when $u$ describes the unit circle in $T_pS$.
\end{defn}

The characteristic curve can be obtained from the indicatrix through a standard transformation in projective geometry:

\begin{defn}
The \emph{pole }of a line $l$ with respect to a conic $C$ is the intersection of the tangents to $C$ at the points of intersection of $l$ with $C$; the \emph{polar} of a point $P$ with respect to a conic $C$ is the line defined by the tangency points of the two tangents to $C$ passing through $P$. The \emph{polar conjugate} of a conic $C'$ with respect to a conic $C$ is the locus of the poles of the tangents to $C'$.
\end{defn}

In particular, the pole of a tangent to $C$ with respect to $C$ is the tangency point, and the polar of that tangency point is the tangent.

\begin{rem}
In the real case, if the point $P$ is inside an ellipse $C$ there are no (real) tangents to $C$ passing through $P$, as there can be no intersection of a line $l$ with the ellipse; using the general fact that the poles of lines all intersecting at $P$ are points in the polar line of $P$, as the polars of the points in a line $l$ are lines all intersecting at the pole of $l$, it is possible to give a geometric construction even for these cases (Fig.~\ref{pol}).
\end{rem}

\begin{figure}[hhhh]
\begin{center}
\includegraphics[width=0.5 \linewidth]{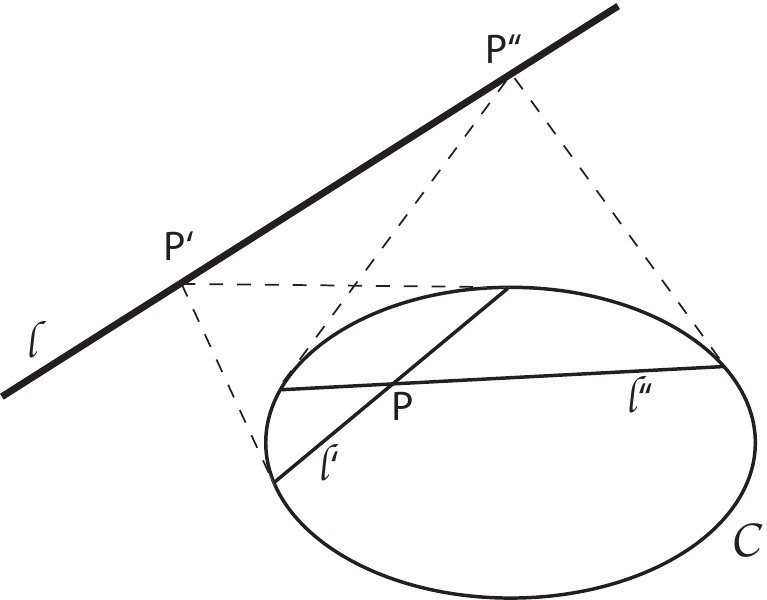}
\end{center}
\caption{$P$ is the pole of $l$; $l'$ and $l''$ are the polars of $P'$ and $P''$}\label{pol}
\end{figure}

\begin{prop}
The characteristic curve is the evolvent of the polars of the points in the indicatrix.
\end{prop}

\begin{figure}[thhh]
\begin{center}
\includegraphics[width=0.75 \linewidth]{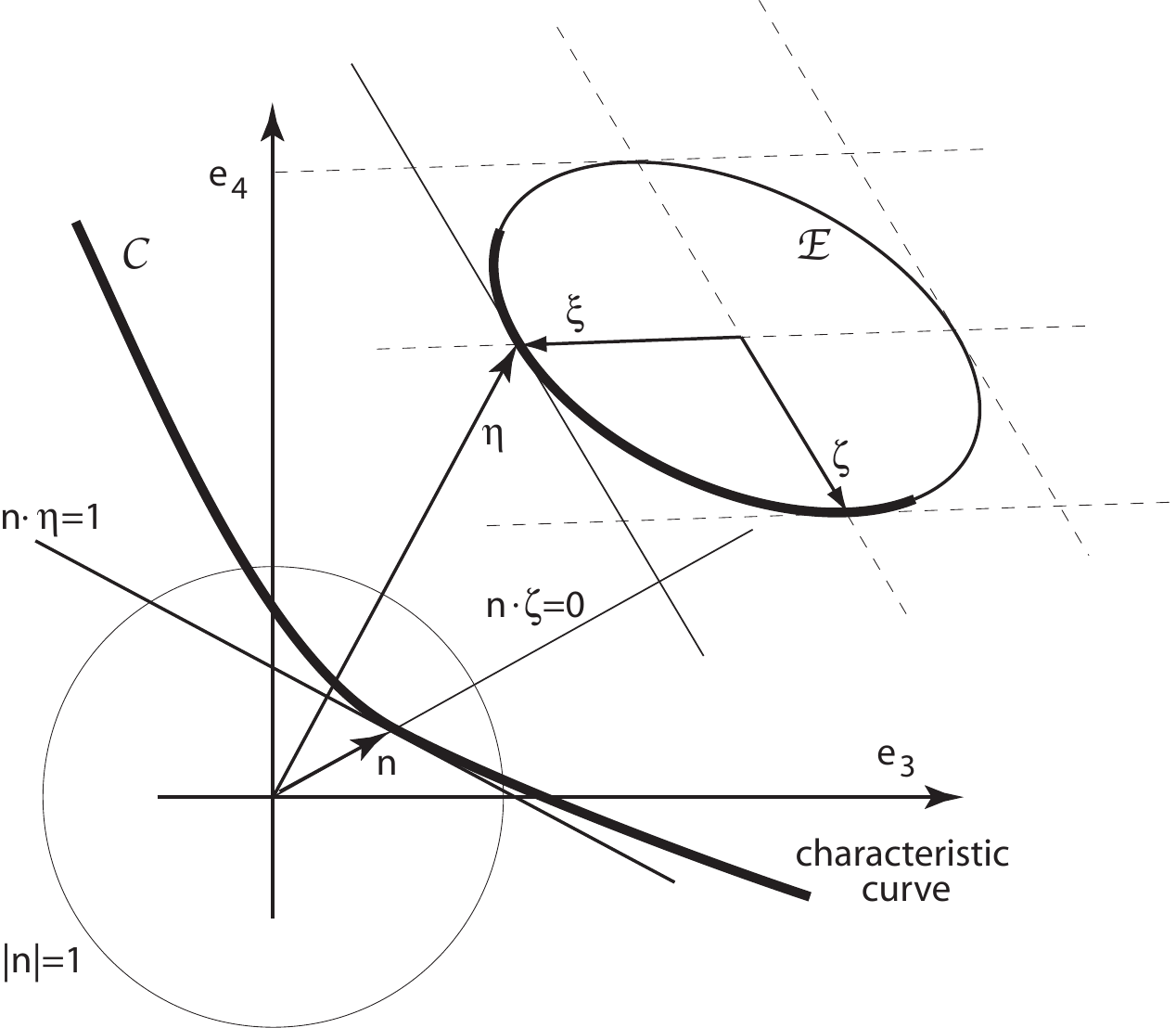}
\end{center}
\caption{Characteristic curve as the evolvent of the family $n\cdot \sff(u)=n\cdot \eta=1$}\label{charcev}
\end{figure}

\begin{proof}

The polar of a point $\eta$  with respect to the unit circle centred at the origin is the line $\eta\cdot n=1$. We have to prove that the characteristic curve is the evolvent of the family of lines $n\cdot \sff(u)=n\cdot \eta=1$ in the normal space $N_pS$.

Let $n(u)$ be the point of contact  of the line $n\cdot \sff(u)=n\cdot \eta=1$ with the characteristic curve $\mathcal C$, where $n\cdot \zeta=0$; we have to prove that it is a tangency point, and this is equivalent to prove that:
\[
\dfrac{d}{du}n(u) \perp \eta
\]

Deriving $=n\cdot \eta=1$ with respect to $u$:
\[
\dfrac{d}{du}n(u) \cdot \eta +n\cdot \dfrac{d}{du}\eta\equiv 0
\]
and as:
\[
\dfrac{d}{du}\eta\parallel \zeta, \qquad n\cdot \zeta=0
\]
it follows that:
\[
\dfrac{d}{du}n(u) \cdot \eta \equiv 0
\]
\end{proof}

\begin{theorem}[Kommerell\cite{kom}]
The characteristic curve is the polar conjugate of the indicatrix or  curvature ellipse with respect to the origin. It is an ellipse, a parabola or a hyperbola  as the point is elliptic, parabolic or hyperbolic.
\end{theorem}

\begin{proof}
When the conic $\mathcal C$ is a circumference, the pole of a line with repeat to $\mathcal C$ is the inverse with respect to the circumference of the foot of the perpendicular from the centre of the circumference to the line. Thus the polar conjugate of the indicatrix  with respect to the unit circle is the inverse  with respect to that circle of the pedal curve of the curvature ellipse.

Let $\eta=\eta(\theta)=\sff(u)$ be the point on the indicatrix $\mathcal E$, or curvature ellipse, corresponding to $u=(\cos\theta,\sin\theta)$, and let $\zeta$ and $\xi$ be conjugate radii of $\mathcal E$ as before, with $\eta=\mathcal H+\xi$. Then $\zeta$ is parallel to the tangent to $\mathcal E$ at $\eta$.

The pedal curve of $\mathcal E$ (with respect to the origin) can be written as $\rho=\rho(\theta)$ with:
\[
\rho=\eta+\dfrac{\eta\cdot\zeta}{|\zeta|^2}\zeta
\]

From the definition of pedal curve, the locus of the intersection of a tangent to the curve with its normal line passing through the origin, we have $\rho\cdot \zeta=0$.

It is easy to see that:
\[
|\rho|^2=|\eta|^2-\left|\dfrac{\eta\cdot\zeta}{|\zeta|^2}\zeta\right|^2=|\eta|^2-|\eta|^2\cos ^2\tau, \qquad \cos\tau=\dfrac{\eta\cdot\zeta}{|\eta||\zeta|}
\]
and therefore $|\rho|=|\eta||\sin\tau|$; we also have $\rho\cdot\eta=|\rho|||\eta||\sin\tau|$.

Thus the inverse of the pedal curve with respect to the unit circle is given by:
\[
n=\dfrac{1}{|\rho|^2}\rho
\]
and we have $n\cdot \zeta=0$. It is only necessary to show that $n\cdot\eta=1$:
\[
n\cdot\eta=\dfrac{1}{|\rho|^2}\rho\cdot\eta=\dfrac{1}{(|\eta||\sin\tau|)^2}|\rho||\eta||\sin\tau|=1
\]

It is well known that the polar conjugate of a conic is another conic.
Concerning the asymptotes, there is a point at infinity in the characteristic curve when the pedal curve passes through the origin, or equivalently when a tangent to the indicatrix passes through the origin; there are 0, 1, or 2 such points when the origin is inside, on or outside the indicatrix.
\end{proof}

\begin{figure}[hhhh]
\begin{center}
\includegraphics[width=0.75 \linewidth]{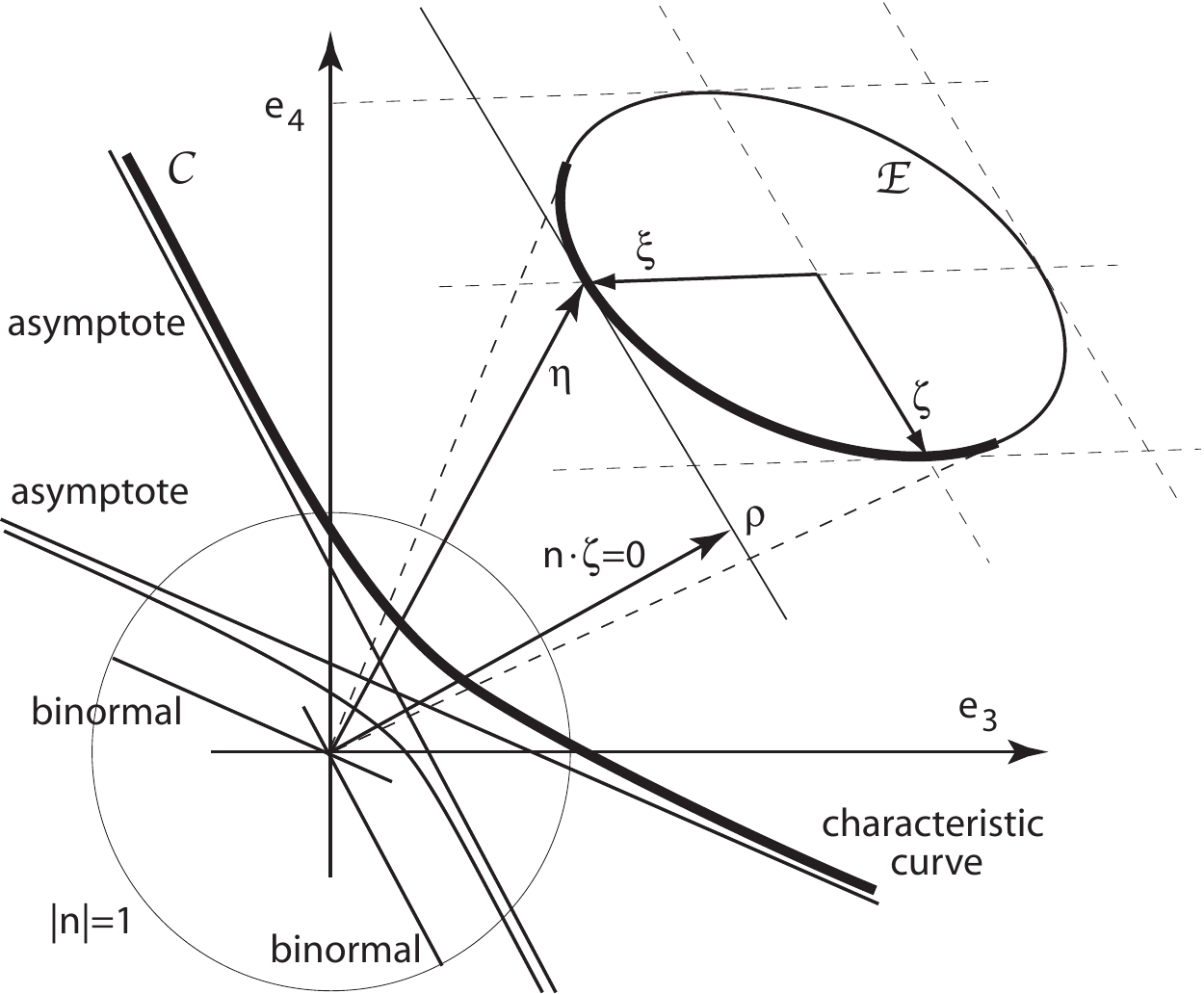}
\end{center}
\caption{Characteristic curve at a hyperbolic point}\label{charc}
\end{figure}

\begin{rem}
The asymptotes are parallel to the respective binormal.
\end{rem}

\begin{rem}
The relation between the proposition and Kommerell theorem is an instance of projective duality: the characteristic curve is the locus of the poles of the tangents to the indicatrix, or  the tangents to the characteristic curve are the polars of the points in the indicatrix.
\end{rem}

\phantom{saltapagina}
\section{Singularities of height functions}\label{shf}

The \emph{height function} on $S$ corresponding to $b\in\mathbf R^4$ is the map  $f_b(p)=f(p,b)$, where:
\[
f:S\times \mathbf R^4\longrightarrow\mathbf R,\quad f(p,b)=p\cdot b
\]

\begin{prop}[\cite{dcc}]
The critical points of $f$ are exactly the points of the normal space $NS$. Moreover:
\begin{itemize}
\item If $\Delta(p)>0$, then $f_b(p)=f(p,b)$ has a non degenerate critical point  at $p$ for all $b\in N_pS$.
\item If $\Delta(p)<0$, then $f_b(p)=f(p,b)$ has a   degenerate critical point at $p$ for exactly two independent normal directions.
\item If $\Delta(p)=0$, then $f_b(p)=f(p,b)$ has a  degenerate critical point at $p$  for exactly one normal direction.
\end{itemize}
\end{prop}

\begin{proof}
The surface $S$ is  locally given around $p$ by a parametrisation:
\[
\Xi : (x,y)\mapsto (x,y,\varphi(x,y),\psi(x,y))
\]
where $\Phi=(\varphi, \psi)$ has vanishing first jet at the origin,  $j^1\Phi(0)=0$, and $\Xi(0)=p$. Then:
\[
f_b(x,y)=b_1x+b_2y+b_3\varphi(x,y)+b_4 \psi(x,y)
\]
having a critical point at the origin implies $b_1=b_2=0$, therefore $b\in N_pS$; we write $b=(0,0,n_1,n_2)$.

The second derivative of $f_b$ is given by:
\[
D^2f_b(x,y)=
\left[
\begin{matrix}
n_1\varphi_{xx}+n_2\psi_{xx} & n_1\varphi_{xy}+n_2\psi_{xy}\\
n_1\varphi_{xy}+n_2\psi_{xy} & n_1\varphi_{yy}+n_2\psi_{yy}
\end{matrix}
\right ]
\]
and so the Hessian of $f_b$ is:
\[
Hess (f_b)=H_{\varphi}n_1^2+Qn_1n_2+H_{\psi}n_2^2
\]

At the origin, the vanishing of:
\[
Hess (f_b)(0)=(ac-b^2)n_1^2+(ag+ce-2bf)n_1n_2+(eg-f^2)n_2^2
\]
is the condition for the critical point to be degenerate. This is a quadratic equation with discriminant (\ref{Delta}):
\[
(ag+ce-2bf)^2-4(ac-b^2)(eg-f^2)=-4\Delta
\]
and the other statements follow.
\end{proof}

The surface $S$ has a higher order contact with the hyperplane normal to $b$ containing the tangent plane to $S$ at $p$, and as remarked in \cite{dcc} this shows that the binormal for a surface in $\Rq$ is an analogue to the binormal of a curve in $\Rt$. 

If the height function $f_b$ has a critical point at $p$, then $f_{\lambda b}$, with $\lambda>0$, has the same type of singular point at $p$; we will consider therefore the height map as being defined on $\mathbf S^3$:
\[
f:S\times \mathbf S^3\longrightarrow\mathbf R,\quad f(p,b)=f_b(p)=p\cdot b
\]
The critical points of $f$ are the points of the unit normal space $N^1S=\{(p,b)\mid p\in S, b\in N_pS, |b|=1\}$. 
If $f_b$ has a degenerate critical point at $p$, in general the kernel of  its second derivative defines a direction in the tangent space $T_pS$ (with the usual identifications), and that is an asymptotic direction:

\begin{prop}\label{binas}
Let $p\in S$ be a degenerate critical point of $f_b$ for some $b\in N_pS$,  $|b|=1$. If the kernel of $D^2f_b$ is one dimensional, it defines an asymptotic direction and $b$ defines a binormal.
\end{prop}
\begin{proof}
We choose coordinates so that $b$ is the fourth axis, or $n_1=0$, $n_2=1$ with the notation of the previous proposition. Furthermore a linear change of coordinates, a rotation around the origin, allows us to assume that the kernel of  its second derivative is the first axis. This means that:
\[
f_b(x,y)=\psi(x,y)=\alpha y^2+O(3), \quad \alpha\ne 0
\]

An unit tangent vector at the origin has the form $v=(\cos\theta, \sin \theta, 0,0)$ and:
\[
\sff (v)=(0,0,\sff_1(v), \alpha\sin^2\theta)
\]
therefore:
\[\eta(0)=\sff(e_1) \parallel e_3, \quad \dfrac{d}{d\theta}\eta (0) \parallel e_3
\]

This means that $\eta(0)$ spans a tangent direction to the curvature ellipse and 
$
b \perp =\eta(0)
$, so that the kernel of the second derivative defines an asymptotic direction and $b$ is a binormal.
\end{proof}

When we consider the contact of a line $l$ with a surface $S$ at a point $p\in S$,  it is clear that the line has to be tangent to the surface at $p$ to have higher order contact. We take the intersection $\gamma$ of the hyperplane through the point containing the line and the normal space $N_pS$; the osculating plane $P_l$  of $\gamma$ at $p$ is defined by the line $l$, tangent to $\gamma$, and the direction spanned by $n _l=\sff(u_l)$, where $u_l$ is a unit vector in $l$. Note that, if the $p$ is not a parabolic point, we have $n_l\ne 0$.

It is natural to say that higher order contact means the vanishing of more derivatives of the component of $\gamma$ orthogonal to the osculating plane. This is equivalent to a higher order singularity of the height function corresponding to a direction normal to the osculating plane, and thus leads to an asymptotic direction.

We could also project $S$ on $\Rt$ along a  normal direction orthogonal to a binormal to obtain a smooth surface $\hat S\subset \Rt$; with coordinates chosen as before, this is the graph of the $\psi$. It is easy to see that the projection of the asymptotic direction corresponding to the chosen binormal is an asymptotic direction (in the usual sense for surfaces in $\Rt$) of $\hat S$:

Since $\hat S$ is the graph of $\psi(x,y)=y^2+O(3)$, it has an asymptotic direction, the $x$ axis. The asymptotic direction of $S$ in $\Rq$ is also the $x$ axis, as seen in prop. \ref{binas}.

\begin{prop}[\cite{dcc}]
Let $p\in S$ be a parabolic point. If $p$ is not an  inflection point then it is a fold or cusp (or higher order) singularity of the height function and:
\begin{itemize}
\item $p$ is a fold singularity when the asymptotic direction is not tangent to the line $\Delta^{-1}(0)$ of parabolic points.
\item $p$ is a cusp (or hgher order) singularity when the asymptotic direction is  tangent to the line $\Delta^{-1}(0)$ of parabolic points.
\end{itemize}
\end{prop}
\begin{proof}
We use the coordinates of prop. \ref{binas}; then $p$ being a parabolic point means $\Delta(0)=0$, and as the asymptotic direction is the $x$ axis, the asymptotic direction being tangent to the line of parabolic points means that $\Delta_x(0)=0$.

Since:
\[
\Delta=(ac-b^2)(eg-f^2)-\dfrac{1}{4}(ag+ce-2bf)^2
\]
at the origin we have:
\[
\Delta(0)=H_\varphi H_\psi-\dfrac{1}{4}(\varphi_{xx}\psi_{yy}+\varphi_{yy}\psi_{xx}-2\varphi_{xy}\psi_{xx})^2
\]

As $H_\psi(0)=0$, we must have $\varphi_{xx}\psi_{yy}+\varphi_{yy}\psi_{xx}-2\varphi_{xy}\psi_{xx}=0$,  therefore:
\[
\Delta_x(0)=H_\varphi H_{\psi,x}
\]
and as $H_\varphi(0)\ne 0$, the condition for the asymptotic direction being tangent to the line of parabolic points becomes:
\[
H_{\psi,x}(0)=0
\]

The point $p$ is not an inflection point, so $\alpha\ne 0$, and the condition for being a fold singularity of $f_b=\psi$ is:
\[
\psi_{xx}=0, \quad \psi_{xxx}\ne 0
\]
and for being a cusp (or higher order) singularity is $\psi_{xxx}= 0$. On the other hand:
\[
H_{\psi,x}(0)=\psi_{yy}(0)\psi_{xxx}(0)=2\alpha\psi_{xxx}(0)
\]

Thus the singularity of $f_b=\psi$ is a cusp if:
\[
H_{\psi,x}(0)\ne 0
\]
and a cusp (or higher order singularity) if:
\[
H_{\psi,x}(0)=0
\]
\end{proof}

\begin{prop}[\cite{dcc}]
The inflection points of a surface correspond to umbilic singularities, or higher singularities, of the height function.
\end{prop}
\begin{proof}
With the coordinates of the previous proposition, the singularity is an umbilic, or more degenerate, if:
\[
\psi_{xx}(0)=0, \quad \psi_{xy}(0)=0, \quad \psi_{yy}(0)=0
\]
therefore::
\[
\mathcal M(p)=\left[
\begin{matrix}
a&b&c\\
0&0&0
\end{matrix}
\right]
\]
and $p$ is an inflection point.

Assume now that $p$ is an inflection point; then $\rank \mathcal M(p)=1$ and there exist $\lambda_1$, $\lambda_2$ such that:
\[
\lambda_1(a,2b,c)+\lambda_2(e,2f,g)=0
\]
and so:
\[
H_{\lambda_1\varphi+\lambda_2\psi}(0)=0
\]
But this means that the height function $f_b$, with $b=(\lambda_1,\lambda_2)$, has an umbilic (or higher order) singularity.
\end{proof}

 The singularities of the family of height functions on a generic surface can be used \cite{dcc} to characterize the different points of that surface:
\begin{itemize}
\item An elliptic point $p$ is a nondegenerate critical point for any 
of the height functions associated to normal directions to $S$ at $p$. 
\item If $p$ is a hyperbolic point, there are exactly 2 normal directions 
 at $p$ such that $p$ is a degenerate critical point of their corresponding 
height functions. 
\item If $p$ is a parabolic point, there is a unique normal direction such that 
$f_b$ is degenerate at $p$.
\begin{itemize}
\item A parabolic point  $p$ is  a fold singularity of $f_b$ if and only if  the unique asymptotic direction is not tangent to the line of parabolic points $\Delta^{-1}(0)$. 
\item A parabolic point  $p$ is  a cusp singularity of $f_b$ if and only if  $p$ is a \emph{parabolic cusp} of 
$S$, where the asymptotic direction is tangent to the line of parabolic points. 
\item A parabolic point  $p$ is  an umbilic point for $f_b$ if and only if  $p$ is an inflection point of 
$S$. 
\end{itemize}
\end{itemize}

\begin{rem}
For a generic surface, the points $p$ which are a swallowtail singularity of $f_b$ do not belong to the line of parabolic points; at a swallowtail singularity one of the asymptotic directions is tangent to line of points where $f_b$ has a cusp singularity.
\end{rem}

\end{document}